\documentclass[11pt,reqno]{amsart}
\usepackage{amsmath, amssymb, amsthm}
\usepackage{url}
\usepackage{amsthm}
\usepackage{lipsum}
\usepackage[breaklinks]{hyperref}
\renewcommand{\(}{\left\(}
\renewcommand{\)}{\right\)}
\renewcommand{\[}{\left\[}
\renewcommand{\]}{\right\]}

\numberwithin{equation}{section}
 \theoremstyle{plain}
\newtheorem{theorem}{Theorem}[section]

\newtheorem{remark}[]{Remark}

\newtheorem{conjecture}[theorem]{Conjecture}

\newtheorem{corollary}[theorem]{Corollary}
\newtheorem{proposition}[theorem]{Proposition}

   \makeatletter
\def\proof{\@ifnextchar[{\@oproof}{\@nproof}}
\def\@oproof[#1][#2]{\trivlist\item[\hskip\labelsep\textit{#2 Proof of\
#1.}~]\ignorespaces}
\def\@nproof{\trivlist\item[\hskip\labelsep\textit{Proof.}~]\ignorespaces}

\makeatother

\begin{document}
\title{An inequality between finite analogues of rank and crank moments} 

\author{Pramod Eyyunni, Bibekananda Maji and Garima Sood}
\address{Discipline of Mathematics, Indian Institute of Technology Gandhinagar, Palaj, Gandhinagar 382355, Gujarat, India} 
\email{pramodeyy@gmail.com, bibek10iitb@gmail.com\\garimasood18@gmail.com}
\maketitle
\begin{center}
{\it
Dedicated to Professor Bruce C. Berndt on the occasion of his 80th birthday}
\end{center}
\begin{abstract}

The inequality between rank and crank moments was conjectured and later proved by Garvan himself in 2011. Recently, Dixit and the authors introduced finite analogues of rank and crank moments for vector partitions while deriving a finite analogue of Andrews' famous identity for smallest parts function. In the same paper, they also conjectured an inequality between finite analogues of rank and crank moments, analogous to Garvan's conjecture. In the present paper, we give a proof of this conjecture. 

\end{abstract}

\section{Introduction}
Let $p(n)$ denote the number of unrestricted partitions of a positive integer $n$. To give a combinatorial explanation of the famous congruences of Ramanujan for the partition function $p(n)$, namely, for $ m\geq 0,$
\begin{align*}
p(5m + 4) & \equiv 0 \pmod 5,\\
p(7m + 5) & \equiv 0 \pmod 7,
\end{align*}
Dyson \cite{dys} defined the rank of a partition as the largest part minus the number of parts. He also conjectured that there must be another statistic, which he named `crank', that would explain Ramanujan's third congruence, namely,
 \begin{equation*}
 p(11m + 6) \equiv 0 \pmod{11}.
 \end{equation*}
 After a decade, Atkin and Swinnerton-Dyer \cite{atkin2} confirmed Dyson's observations for the first two congruences for $p(n)$. Also, in 1988, `crank' was discovered by Andrews and Garvan \cite{andrewsgarvan88}. 
An interesting thing to note is that by using the partition statistic `crank', Andrews and Garvan were able to explain not only the third congruence but also the first two. 
Atkin and Garvan \cite{atkin1} found that the moments of ranks and cranks were important in the study of further partition congruences. In particular, they defined the ${k}^{th}$ moments of rank and crank, respectively as, 
\begin{align*}
N_{k}(n)=\sum_{m=-\infty}^{\infty} m^{k} N(m,n),\\
M_{k} (n) = \sum_{m=-\infty}^{\infty}m^{k} M(m, n),
\end{align*}
where $N(m,n)$ and $M(m,n)$ denote the number of partitions of $n$ with rank $m$ and crank $m$ respectively. 
In 2008, Andrews \cite{andrews08} introduced the smallest parts function $\mathrm{spt}(n)$ as the total number of appearances of the smallest parts in all partitions of $n$ and showed that
\begin{equation*}
\text{spt}{(n)}=np(n)-\frac{1}{2}N_2(n).
\end{equation*} 
Using Dyson's identity \cite[Theorem 5]{dyson89}, i.e., 
$np(n)=\frac{1}{2}M_2(n) $, we can rewrite this as
\begin{equation}\label{Andrews_spt_identity}
\text{spt}{(n)}=\frac{1}{2}M_2(n)-\frac{1}{2}N_2(n).
\end{equation}
From this result, it is immediate that $M_2(n)>N_2(n)$.
Garvan \cite[Conjecture (1.1)]{garvan10} conjectured  that
\begin{equation}\label{garvan's conjecture}
M_{2k} (n)>N_{2k} (n),
\end{equation}
for all $k>1$ and $n\geq1$. 
Studying the asymptotic behavior of the difference $M_{2k}(n)-N_{2k}(n)$, Bringmann and Mahlburg \cite{BM09} proved \eqref{garvan's conjecture} for $k=2,4$, and subsequently, for each fixed $k$, the  inequality was proved for sufficiently large $n$ by  Bringmann, Mahlburg and Rhoades \cite{BMR11}.
Later, Garvan \cite{garvan10} himself proved
his conjecture for all $n$ and $k$ with the help of a combinatorial interpretation for the difference
between symmetrized crank and rank moments.
Andrews \cite{andrews07} defined the $k^{th}$ symmetrized rank moment as
\begin{equation*}
\eta_{k}(n):=\sum_{m=-n}^{n} \left(\begin{matrix} m+ \lfloor \frac{k-1}{2} \rfloor \\ k\end{matrix}\right)N(m,n).
\end{equation*}
Andrews \cite[Theorem 2]{andrews07} showed that the odd moments are all identically zero and also obtained the generating function for even moments $\eta_{2k}(n)$, that is,
for any $k \geq 1$, we have
\begin{align}
\sum_{n=1}^\infty \eta_{2k}(n) q^n & = \frac{1}{(q)_\infty} \sum_{n=1}^\infty \frac{(-1)^{n-1}(1+q^n) q^{\frac{n(3n-1)}{2}+kn}}{(1-q^n)^{2k}} \label{Andrews_genfn_symmetrized_rank}\\
& = \frac{1}{(q)_\infty} \sum_{\substack{ n=-\infty\\ n\neq 0}}^\infty \frac{(-1)^{n-1} q^{\frac{n(3n+1)}{2}+kn}}{(1-q^n)^{2k}}.\label{andrews_symm_rank2}
\end{align}
Analogous to the symmetrized rank moments $\eta_k(n)$, Garvan \cite{garvan11} introduced the $k^{th}$ symmetrized crank moment $\mu_k(n)$ in the study of the higher order
$\text{spt}$-function $\text{spt}_k(n)$. To be more specific,
\begin{align*}{}
\mu_{k}(n):=\sum_{m=-n}^{n} \left(\begin{matrix} m+ \lfloor \frac{k-1}{2} \rfloor \\ k\end{matrix}\right)M(m,n). 
\end{align*}
Analogous to \eqref{Andrews_genfn_symmetrized_rank} and \eqref{andrews_symm_rank2}, the generating function for the symmetrized crank moments was given by Garvan \cite[Theorem (2.2)]{garvan11}, that is, 
for any $k \geq 1$, we have
\begin{align} 
\sum_{n=1}^{\infty} \mu_{2k}(n) q^n & =  
\frac{1}{(q)_\infty }  \sum_{n=1}^\infty \frac{(-1)^{n-1}(1+q^n) q^{\frac{n(n-1)}{2}+kn}}{(1-q^n)^{2k}} \label{garvan_symmcrank_1}\\
& = \frac{1}{(q)_\infty} \sum_{\substack{ n=-\infty\\ n\neq 0}}^\infty \frac{(-1)^{n-1} q^{\frac{n(n+1)}{2}+kn}}{(1-q^n)^{2k}}.\label{garvan_symmcrank_2}
\end{align}
Garvan \cite[Equation (1.4)]{garvan11} also gave the following generating function for the symmetrized crank moments:
\begin{align}\label{mugf_alternate}
\sum_{n=1}^{\infty} \mu_{2k}(n) q^n =  
\frac{1}{(q)_\infty }\sum_{ n_k\geq  \cdots \geq n_2\geq1}\frac{q^{n_1+n_2+ \cdots +n_k} }{(1-q^{n_1})^2(1-q^{n_2})^2 \cdots (1-q^{n_k})^2}.
\end{align}
One of the main results in \cite{garvan11}, due to Garvan \cite[Equation (1.3)]{garvan11}, which was instrumental in proving the inequality between rank and crank moments is as follows: 
\begin{align}\label{difference_gen_garvan}
\sum_{n=1}^{\infty}( \mu_{2k}(n)- \eta_{2k}(n)) q^n = \sum_{ n_k\geq  \cdots \geq n_2\geq1}\frac{q^{n_1+n_2+ \cdots +n_k} }{(1-q^{n_1})^2(1-q^{n_2})^2 \cdots (1-q^{n_k})^2 (q^{n_1+1} ; q)_\infty},
\end{align}
for any $k \geq 1$.
One can easily check that for $k=1$, the above theorem reduces to \eqref{Andrews_spt_identity}. After this observation, Garvan defined higher order $\text{spt}$-function $\text{spt}_k(n)$ as
\begin{equation*}
\text{spt}_k(n) := \mu_{2k}(n)-\eta_{2k}(n),
\end{equation*} 
for all $k\geq1$ and $n\geq1$.
He also gave a combinatorial interpretation of $\text{spt}_{k}(n)$. 

In the next subsection, we shall describe recent developments related to Andrews' identity \eqref{Andrews_spt_identity} for the smallest parts function $\text{spt}(n)$. 
\subsection{Finite analogue of Andrews' spt-identity}
Ramanujan's identities are a constant source of inspiration for everyone and motivate us to do beautiful mathematics. 
Recently, Dixit and Maji \cite{dixitmaji18} found a generalization of a $q$-series identity \cite[p.~354]{ramanujantifr}, \cite[p. 263, Entry 3]{bcbramforthnote} of Ramanujan and derived many partition theoretic implications from this generalization. They also established a new identity \cite[Theorem 2.8]{dixitmaji18} involving Fine's function $F(a,b;t)$ \cite[p.~1]{fine} from which they were able to derive Andrews' identity \eqref{Andrews_spt_identity} for $\text{spt}(n)$. Very recently, together with Dixit, the authors found a finite analogue \cite[Theorem 1.1]{DEMS} of the aforementioned generalization of Dixit and Maji \cite[Theorem 2.1]{dixitmaji18}, whose special case gave a finite analogue of Andrews' $\text{spt}$-identity, namely,
\begin{theorem}\cite[Theorem 2.4]{DEMS}\label{finite analogue_Andrews_spt}
 For any natural numbers $n, N$, we have
 \begin{align*}
 \textup{spt}(n, N)= \frac{1}{2} \left( M_{2,N}(n) - N_{2,N}(n) \right),
 \end{align*}
 where $\textup{spt}(n, N )$ is the number of smallest parts in all partitions of $n$ whose corresponding largest parts are less than or equal to $N$, and $M_{2,N}(n)$ and $N_{2,N}(n)$ are defined below.
\end{theorem}
In \cite[p.~9, Equations (2.9), (2.10)]{DEMS}, for $k\geq 1$, we defined finite analogues of rank and crank moments for vector partitions as 
\begin{align}
N_{k, N}(n) &:= \sum_{m=-\infty}^{\infty} m^k N_{S_1}(m, n),\label{finrankmom}\\
M_{k,N}(n)&:= \sum_{m=-\infty}^{\infty} m^k M_{S_2}(m, n),\label{fincrankmom}
\end{align}
where $N_{S_1}(m, n)$ and $M_{S_2}(m, n)$ are defined below in \eqref{ns1} and \eqref{ms1} respectively. 
From Theorem \ref{finite analogue_Andrews_spt}, it is immediate that $M_{2,N}(n) > N_{2,N}(n)$.  Analogous to Garvan's conjecture \eqref{garvan's conjecture},  we gave the following conjecture on the inequality between the finite analogues of $k^{\textup{th}}$ rank and crank moments, that is, 
\begin{conjecture}\cite[Conjecture 10.1]{DEMS}\label{finiteanalogcrankrankconjecture}
For any fixed natural number $N$ and even $k\geq 2$, 
\begin{equation*}
M_{k,N}(n) > N_{k,N}(n) \quad \mathrm{for \,\, all} \,\, n\geq 1. 
\end{equation*}
\end{conjecture}
In the present paper, our main goal is to prove the above conjecture. We have already mentioned in this introduction that the theory of symmetrized rank and crank moments was developed by Andrews \cite{andrews07} and Garvan \cite{garvan11} respectively. Here, to prove the above conjecture we define finite analogues of symmetrized rank and crank moments and their generating functions. We follow similar techniques as employed by Garvan \cite{garvan11}.

\section{Main Results: Finite analogues of the $k^{th}$ symmetrized rank and crank moments}\label{S2}
Before defining finite analogues of symmetrized rank and crank moments we need to recall certain definitions from \cite[p.~7]{DEMS}. For the sake of completeness we reproduce them below.

Let $V_{1}=\mathcal{D}\times\mathcal{P}$ denote a set of vector partitions. So an element $\vec{\pi}$ of $V_1$ is of the form $(\pi_1, \pi_2)$, where the magnitude of $\vec{\pi}$ is 
given by $|\vec{\pi}|:=|\pi_1|+|\pi_2|$. 
Let $N$ be a positive integer. Then for any positive integer $j$ with $1\leq j \leq N$, set
\begin{align*}\label{S1}
S_{1}&:=\bigg\{\vec{\pi}\in V_1: \pi_1\hspace{0.5mm}\text{is either an empty partition or such that its parts lie in $[N-j+1, N]$ }\nonumber \\  & \quad\quad \text{and}\,\, \pi_2\hspace{1mm}\text{is an }
\text{unrestricted partition with Durfee square of size $j$}\bigg\}.
\end{align*}
For a vector partition $\vec{\pi}=(\pi_1, \pi_2)$ in $V_1$, let $w_r(\vec{\pi}) := (-1)^{\#(\pi_1)}$ be its weight and 
$\textup{rank}(\vec{\pi}):=\textup{rank}(\pi_2)$, its vector rank. Now define 
\begin{equation}\label{ns1}
N_{S_1}(m, n):=\sum_{j=1}^N N_{S_{1}}\left(m, n; \boxed{j}\right),
\end{equation}
where
\begin{equation*}
N_{S_{1}}\left(m, n; \boxed{j}\right):= \sum_{\vec{\pi} \in S_{1}, |\vec{\pi}|=n \atop \mathrm{rank}(\vec{\pi})=m} w_r(\vec{\pi}).
\end{equation*}
As observed in \cite{DEMS}, as $N \rightarrow \infty$, $N_{S_1}(m, n)$ equals $N(m, n)$, the number of ordinary partitions of $n$ with rank $m$.

We are now ready to define the finite analogue of the $k^{th}$ symmetrized rank function. Let $k, N$ be positive integers. Then for any $n \geq 1$,
\begin{equation}\label{FR}
\eta_{k,N}(n):=\sum_{m=-n}^{n} \left(\begin{matrix} m+ \lfloor \frac{k-1}{2} \rfloor \\ k\end{matrix}\right)N_{S_1}(m,n).
\end{equation} 
\begin{proposition} Let $N$ be a positive integer and $k$ be an odd positive integer. 
Then $\eta_{k,N}(n)=0$ for all $n\geq1$.
\end{proposition}
This is straightforward from the fact that the finite analogues of all the odd rank moments $N_{k, N}(n)$ are zero. We now give an expression for the generating 
function of $\eta_{k,N}(n)$ for \emph{even} $k$.
\begin{theorem}\label{gffinsymmetrizedrankfn}
Let $N \in \mathbb{N}$. Then for any positive integer $\nu$, we have
\begin{align}
\sum_{n=1}^{\infty}\eta_{2\nu,N}(n)q^n&=(q)_N \sum_{\substack {n=1}}^{N}\frac{(-1)^{n-1}  q^{\frac{n(3n-1)}{2}+\nu n}(1+q^n)  }{(q)_{N+n}(q)_{N-n}(1-q^{n})^{2\nu}} \label{gf_symm_rank2}\\
 &=(q)_N\sum_{\substack {n=-N\\n\neq 0}}^{N}\frac{(-1)^{n-1}  q^{\frac{n(3n+1)}{2}+\nu n}  }{(q)_{N+n}(q)_{N-n}(1-q^n)^{2\nu}}. \label{gfsymrankmoment} 
\end{align}
\end{theorem}
Letting $N \rightarrow \infty$, we obtain the generating functions for the symmetrized rank moment, namely, \eqref{Andrews_genfn_symmetrized_rank} and \eqref{andrews_symm_rank2}. 
Next, we are going to define the finite analogue of  the symmetrized crank moments.
Again, for convenience, we recollect some definitions from \cite[p.~8-9]{DEMS}. 

Let $V_{2}$ denote the set of vector partitions $\mathcal{D}\times\mathcal{P}\times\mathcal{P}$. Denote an element $\vec{\pi}$ of $V_2$ by 
$(\pi_1, \pi_2, \pi_3)$ so that the magnitude of $\vec{\pi}$ is $|\vec{\pi}|=|\pi_1|+|\pi_2|+|\pi_3|$.

For any positive integer $N$, we define the following set:
\begin{align*}{}
S_2:=\{\vec{\pi}\in V_2 : l(\pi_1), l(\pi_2), l(\pi_3) \leq N \}. 
\end{align*}
Define $w_c(\vec{\pi}):=(-1)^{\#(\pi_1)}$ to be the weight of the vector partition $\vec{\pi}=(\pi_1, \pi_2, \pi_3)$ 
and crank$(\vec{\pi}):= \#(\pi_2)-\#(\pi_3)$ be its vector crank. We define
\begin{equation}\label{ms1}
M_{S_2}(m, n) := \sum_{\vec{\pi} \in S_2, |\vec{\pi}|=n \atop \mathrm{crank}(\vec{\pi})=m} w_c(\vec{\pi}).   
\end{equation}
Letting $N \rightarrow \infty$ we see that $S_2$ approaches the whole set $V_2$ and consequently $M_{S_2}(m, n)$ approaches 
$\sum_{\vec{\pi} \in V_2, |\vec{\pi}|=n \atop \mathrm{crank}(\vec{\pi})=m} w_c(\vec{\pi})$, which is the total number of weighted vector partitions of $n$ with vector
crank $m$, a quantity first studied by Garvan (See \cite[p.~50]{garvan88}). By the work of Andrews and Garvan \cite[Theorem 1]{andrewsgarvan88}, 
we know that this equals $M(m, n)$, the number of integer partitions of $n$ with crank $m$. 

We now define a finite analogue of the $k^{th}$ symmetrized crank moment. Let $k, N$ be positive integers. Then for any $n \geq 1$,
\begin{equation}\label{SCM}
\mu_{k,N}(n):=\sum_{m=-n}^{n} \left(\begin{matrix} m+ \lfloor \frac{k-1}{2} \rfloor \\ k\end{matrix}\right)M_{S_2}(m,n).
\end{equation}

\begin{proposition} 
For any odd positive integer $k$, we have $\mu_{k,N}(n)=0$ for all $n\geq1$. 
\end{proposition}
This easily follows because all the odd crank moments $M_{k, N}(n)$ are zero. Analogous to Theorem \ref{gffinsymmetrizedrankfn} above, we derive the following result for the generating function of $\mu_{k, N}(n)$ for \emph{even} $k$.

\begin{theorem}\label{theoremmugf}
 Let $N \in \mathbb{N}$. Then for any positive integer $\nu$, one has
\begin{align}
\sum_{n=1}^{\infty}\mu_{2\nu,N}(n)q^n &=(q)_N \sum_ {n=1}^{N}\frac{(-1)^{n-1} q^{\frac{n(n-1)}{2}+\nu n}(1+q^n)}{(q)_{N+n}(q)_{N-n}(1-q^n)^{2\nu}} \label{mugf} \\
 &= (q)_N\sum_{\substack {n=-N\\n\neq 0}}^{N}\frac{(-1)^{n-1}  q^{\frac{n(n+1)}{2}+\nu n}  }{(q)_{N+n}(q)_{N-n}(1-q^n)^{2\nu}} \label{fin_symm_crank_2}
\end{align}
\end{theorem}
One can easily observe that this result is a finite analogue of the equations \eqref{garvan_symmcrank_1} and \eqref{garvan_symmcrank_2} 
by letting $N\rightarrow \infty$.
The next result provides us information about the generating function of the difference between finite analogues of symmetrized crank and rank moments. 
\begin{theorem}\label{difference_gen_fin_rank_crank}
Let $N \in \mathbb{N}$. Then for any positive integer $k$, we have
\begin{equation}\label{gen_fin_rank_crank}
\sum_{n=1}^{\infty}(\mu_{2k,N}(n)-\eta_{2k,N}(n))q^n=\frac{1}{(q)_N}\sum_{N\geq n_k\geq ...\geq n_1\geq1}\frac{(q)_{n_1}q^{n_1+n_2+...+n_k}}{(1-q^{n_1})^2(1-q^{n_2})^2...(1-q^{n_k})^2}.
\end{equation}
\end{theorem}
This is a finite analogue of Garvan's result \eqref{difference_gen_garvan} for the generating function of the difference between symmetrized crank and rank moments.
\begin{remark}\label{remark_fin_higher_spt}
If we substitute $k=1$ in the above result, then we can obtain Theorem \textup{\ref{finite analogue_Andrews_spt}}. Thus we have $\mu_{2,N}(n)-\eta_{2,N}(n)=\textup{spt}(n,N)$. This suggests us to define a finite analogue of higher order spt-function as 
$\textup{spt}_k(n,N):= \mu_{2k,N}(n)-\eta_{2k,N}(n)$. 
\end{remark}
The remainder of this paper is organized as follows. In the next section we collect all necessary results which will be useful throughout the paper. The generating functions of the finite analogues of the symmetrized rank and crank moments are proved in Section \ref{main_theorems}. In Section \ref{baily_pairs_proof_conj}, we derive important results using Bailey's lemma and give a proof of Conjecture \ref{finiteanalogcrankrankconjecture}. 
We conclude the paper, by discussing further questions in Section 6.  
\section{Preliminaries}
In \cite[Theorem 2.2]{DEMS}, Dixit et al. noted that the generating function of $N_{S_1}(m, n)$ is 
\begin{align}\label{RS1zqdefinition}
R_{S_1}(z;q):=\sum_{n=1}^{\infty}\sum_{m=-\infty}^{\infty}N_{S_1}(m,n)z^mq^n=\sum_{j=1}^N \left[\begin{matrix} N \\ j \end{matrix}\right] \frac{q^{j^2} (q)_j }{(z q)_j (z^{-1} q)_j}.
\end{align}
We call \eqref{RS1zqdefinition} as the finite analogue of the rank generating function, for, letting $N \rightarrow \infty$ on  both sides, gives the well-known result for the rank generating function (for more details, see \cite[p. 8]{DEMS}), 
\begin{equation*}
\sum_{n=1}^{\infty}  \sum_{m=-\infty}^{\infty} N\left(m, n\right) z^m q^n =  \sum_{j=1}^\infty  \frac{q^{j^2}  }{(z q)_j (z^{-1} q)_j}.
\end{equation*}
Again, in \cite[p.~252, Theorem 2.1]{andpar}, \cite[Equation (12.2.2), p.~263]{abramlostI}, Andrews showed that 
\begin{align}\label{frgfbil}
\sum_{n=0}^{N}\left[\begin{matrix} N \\ n \end{matrix}\right]\frac{(q)_nq^{n^2}}{(zq)_n(z^{-1}q)_n}=\frac{1}{(q)_N}+(1-z)\sum_{n=1}^{N}\left[\begin{matrix} N \\ n \end{matrix}\right]\frac{(-1)^n(q)_nq^{n(3n+1)/2}}{(q)_{N+n}}\left(\frac{1}{1-zq^n}-\frac{1}{z-q^n}\right).
\end{align}
Now we recall the crank generating function, that is,
\begin{equation}\label{crank_gen_func}
\frac{(q)_{\infty}}{(zq)_\infty(z^{-1}q)_\infty} = \sum_{n=0}^{\infty} \sum_{m=-\infty}^{\infty} M(m, n) z^m q^n,
\end{equation}
where $M(m,n)$ is the number of partitions of $n$ with crank $m$. 
In \cite[Theorem 2.3]{DEMS}, Dixit et al. proved that the generating function of $M_{S_2}(m,n)$ is 
\begin{equation}\label{fincrankgf}
 C_{S_2}(z;q) := \sum_{n=0}^{\infty}\sum_{m=-\infty}^{\infty}M_{S_2}(m,n)z^mq^n= \frac{(q)_N}{(zq)_N (z^{-1}q)_N},
\end{equation}
which is the finite analogue of \eqref{crank_gen_func}.
Andrews \cite[p.~258, Theorem 4.1]{andpar} showed that
\begin{equation}\label{andrewsCS2}
\frac{(q)_{N}}{(zq)_{N}(z^{-1}q)_{N}}=\frac{1}{(q)_N}+(1-z)\sum_{n=1}^{N}\left[\begin{matrix} N \\ n \end{matrix}\right]\frac{(-1)^n(q)_nq^{n(n+1)/2}}{(q)_{n+N}}\left(\frac{1}{1-zq^n}-\frac{1}{z-q^n}\right).
\end{equation}
Now we collect some useful facts about Bailey pairs, see \cite[p. 582]{andrewsaskey99}.
A pair of sequences $(\alpha_n(a, q), \beta_n(a, q))$ is called a Bailey pair with parameters $(a, q)$ if, for each non-negative integer $n$,
\begin{equation}\label{Baileydefn}
\beta_n(a, q) = \sum_{r=0}^{n}\frac{\alpha_r(a, q)}{(q;q)_{n-r} (aq;q)_{n+r}}.
\end{equation}
\begin{theorem}[Bailey's Lemma]\label{Baileylemma}
Suppose $(\alpha_n(a, q), \beta_n(a, q))$ is a Bailey pair with parameters $(a, q)$. Then $(\alpha'_n(a, q), \beta^{'}_n(a, q))$ is another Bailey pair 
with parameters $(a, q)$, where
$$
\alpha'_n(a, q)=\frac{(\rho_1, \rho_2;q)_n}{(aq/\rho_1, aq/\rho_2;q)_n}\left(\frac{aq}{\rho_1\rho_2} \right)^n \alpha_n(a, q)
$$
and
$$
\beta^{'}_n(a, q)=\sum_{k=0}^{n}\frac{(\rho_1, \rho_2;q)_k (aq/\rho_1\rho_2;q)_{n-k}}{(aq/\rho_1, aq/\rho_2;q)_n (q;q)_{n-k}}\left(\frac{aq}{\rho_1 \rho_2}\right)^k\beta_k(a, q).
$$
\end{theorem}
We also require the following result:
\begin{equation}\label{limitresultp1p2}
\lim_{\rho_2 \rightarrow 1} \lim_{\rho_1 \rightarrow 1}\frac{1}{(1-\rho_1)(1-\rho_2)}\left( 1-\frac{(q)_k (q/\rho_1 \rho_2)_k}{(q/\rho_1)_k(q/\rho_2)_k}\right)=\sum_{j=1}^{k}\frac{q^j}{(1-q^j)^2}.
\end{equation}

\section{Proofs of Theorem \ref{gffinsymmetrizedrankfn} and Theorem \ref{theoremmugf} }\label{main_theorems}
\begin{proof}[Theorem \textup{\ref{gffinsymmetrizedrankfn}}][]
By definition \eqref{FR} of $\eta_{k, N}(n)$, we know that 
\begin{equation*}
\eta_{2\nu,N}(n):=\sum_{m=-n}^{n} \left(\begin{matrix} m+ \nu-1\\ 2\nu\end{matrix}\right)N_{S_1}(m,n).
\end{equation*}
From the definition \eqref{RS1zqdefinition} of $R_{S_1}(z;q)$, it follows at once that
\begin{equation*}
\left(\frac{d^{2\nu}}{dz^{2\nu}}z^{\nu-1}R_{S_1}(z;q)\right)\Big|_{z=1} = \sum_{n=1}^{\infty}\sum_{m=-\infty}^{\infty}(m+\nu-1)(m+\nu-2)\cdots(m-\nu+1)(m-\nu)N_{S_1}(m,n)q^n.
\end{equation*}
In other words, 
\begin{equation*}
\left(\frac{d^{2\nu}}{dz^{2\nu}}z^{\nu-1}R_{S_1}(z;q)\right)\Big|_{z=1} = (2\nu)!\sum_{n=1}^{\infty}\eta_{2\nu,N}(n)q^n.
\end{equation*}
Using Leibniz's chain rule, we get
\begin{align}\label{LBR}
\sum_{n=1}^{\infty}\eta_{2\nu,N}(n)q^n=\frac{1}{(2\nu)!}\sum_{j=0}^{\nu-1}\left(\begin{matrix} 2\nu\\ j\end{matrix}\right)(\nu-1)(\nu-2)...(\nu-j)R_{S_1}^{(2\nu-j)}(1;q).
\end{align}
It will be sufficient for us to find the derivatives of $R_{S_1}(z;q)$ with respect to $z$. To this end, we wish to write $R_{S_1}(z;q)$ in a suitable form. 
Using \eqref{frgfbil} in the right-most expression of \eqref{RS1zqdefinition}, we deduce that
\begin{align*}
R_{S_1}(z;q)&= \frac{1}{(q)_N} - 1 + (1-z)\sum_{n=1}^{N}\left[\begin{matrix} N \\ n \end{matrix}\right]\frac{(-1)^n(q)_nq^{n(3n+1)/2}}{(q)_{N+n}}\left(\frac{1}{1-zq^n}-\frac{1}{z-q^n}\right) \\ 
&= -1 + \frac{1}{(q)_N}\left(1+\sum_{n=1}^{N}\frac{(-1)^n (q)^2_N q^{\frac{n(3n+1)}{2}}}{(q)_{N+n}(q)_{N-n}}\left( \frac{1-z}{1-zq^n}+\frac{1-z^{-1}}{1-z^{-1}q^n}\right)\right).
\end{align*}
Splitting the summation in the right hand side above, we get
\begin{align*}%
1+R_{S_1}(z;q)=\frac{1}{(q)_N}\left(1+\sum_{n=1}^{N}\frac{(-1)^n (q)^2_N q^{\frac{n(3n+1)}{2}}}{(q)_{N+n}(q)_{N-n}} \left(\frac{1-z}{1-zq^n}\right)+\sum_{n=1}^{N}\frac{(-1)^n (q)^2_N q^{\frac{n(3n+1)}{2}} }{(q)_{N+n}(q)_{N-n}} \left(\frac{1-z^{-1}}{1-z^{-1}q^n}\right)\right).
\end{align*} 
Making a change of variable from $n$ to $-n$ in the rightmost summation above, we arrive at
\begin{align*}
1+R_{S_1}(z;q)=\frac{1}{(q)_N}\sum_{n=-N}^{N}\frac{(-1)^n (q)^2_N q^{\frac{n(3n+1)}{2}}}{(q)_{N+n}(q)_{N-n}} \left(\frac{1-z}{1-zq^n}\right).
\end{align*}
We now take the derivatives of $1+R_{S_1}(z;q)$ with respect to $z$. Firstly, we obtain
\begin{align*}{}
R'_{S_1}(z;q)=\frac{-1}{(q)_N}\sum_{\substack {n=-N\\n\neq 0}}^{N}\frac{(-1)^n (q)^2_N q^{\frac{n(3n+1)}{2}}  }{(q)_{N+n}(q)_{N-n}} \left(\frac{1-q^n}{(1-zq^n)^2}\right)
\end{align*}
and so for $j\geq1$,
\begin{align}{\label{RS1}}
R^{(j)}_{S_1}(z;q)=\frac{-j!}{(q)_N}\sum_{\substack {n=-N\\n\neq 0}}^{N}\frac{(-1)^n (q)^2_N q^{\frac{n(3n-1)}{2}+jn}  }{(q)_{N+n}(q)_{N-n}} \left(\frac{1-q^n}{(1-zq^n)^{j+1}}\right).
\end{align}
Putting \eqref{RS1} in the right hand side of \eqref{LBR}, we have
\begin{align*}
\sum_{n=1}^{\infty}\eta_{2\nu,N}(n)q^n &=(q)_N\sum_{j=0}^{\nu-1}\left(\begin{matrix} \nu-1\\ j\end{matrix}\right)\sum_{\substack {n=-N\\n\neq 0}}^{N}\frac{(-1)^{n-1}  q^{\frac{n(3n-1)}{2}+(2\nu-j)n}  }{(q)_{N+n}(q)_{N-n}} \left(\frac{1-q^n}{(1-q^n)^{2\nu-j+1}}\right)\\
&=(q)_N\sum_{\substack {n=-N\\n\neq 0}}^{N}\frac{(-1)^{n-1}  q^{\frac{n(3n-1)}{2}+2\nu n}  }{(q)_{N+n}(q)_{N-n}(1-q^n)^{2\nu}} \sum_{j=0}^{\nu-1}\left(\begin{matrix} \nu-1\\ j\end{matrix}\right)\frac{q^{-nj}}{(1-q^n)^{-j}}\\
&=(q)_N\sum_{\substack {n=-N\\n\neq 0}}^{N}\frac{(-1)^{n-1}  q^{\frac{n(3n+1)}{2}+\nu n}  }{(q)_{N+n}(q)_{N-n}(1-q^n)^{2\nu}},
\end{align*}
by an application of binomial theorem to the inner sum in the second step. Therefore,
\begin{equation*}
\sum_{n=1}^{\infty}\eta_{2\nu,N}(n)q^n=(q)_N\sum_{\substack {n=-N\\n\neq 0}}^{N}\frac{(-1)^{n-1}  q^{\frac{n(3n+1)}{2}+\nu n}  }{(q)_{N+n}(q)_{N-n}(1-q^n)^{2\nu}}.
\end{equation*}
We split the sum on the right side into two parts, namely, from $1$ to $N$ and from $-N$ to $-1$.  
\begin{align*}{}
\sum_{n=1}^{\infty}\eta_{2\nu,N}(n)q^n=(q)_N \left(\sum_{\substack {n=1}}^{N}\frac{(-1)^{n-1}  q^{\frac{n(3n+1)}{2}+\nu n}  }{(q)_{N+n}(q)_{N-n}(1-q^n)^{2\nu}}+\sum_{\substack {n=-1}}^{-N}\frac{(-1)^{n-1}  q^{\frac{n(3n+1)}{2}+\nu n}  }{(q)_{N+n}(q)_{N-n}(1-q^n)^{2\nu}}\right).
\end{align*}
Replace $n$ by $-n$ in the rightmost sum to get
\begin{align*}
\sum_{n=1}^{\infty}\eta_{2\nu,N}(n)q^n=&(q)_N \sum_{\substack {n=1}}^{N}\frac{(-1)^{n-1}  q^{\frac{n(3n-1)}{2}+\nu n}(1+q^n)  }{(q)_{N+n}(q)_{N-n}(1-q^{n})^{2\nu}},
\end{align*}
which is nothing but \eqref{gf_symm_rank2}.
\end{proof}

\begin{proof}[Theorem \textup{\ref{theoremmugf}}][]
We know from \eqref{SCM} that
\begin{equation*}
 \mu_{2\nu,N}(n):=\sum_{m=-n}^{n} \left(\begin{matrix} m+ \nu-1 \\ 2\nu \end{matrix}\right)M_{S_2}(m,n).
\end{equation*}
It follows, from the definition \eqref{fincrankgf} of $C_{S_2}(z;q)$ and by an application of Leibniz's rule, that
\begin{equation}\label{gfmuintermsofCS2}
\sum_{n=1}^{\infty}\mu_{2\nu,N}(n)q^n=\frac{1}{(2\nu)!}\sum_{j=0}^{\nu-1}\left(\begin{matrix} 2\nu\\ j\end{matrix}\right)(\nu-1)(\nu-2)...(\nu-j)C_{S_2}^{(2\nu-j)}(1;q).
\end{equation}
Using \eqref{andrewsCS2} in \eqref{fincrankgf}, we get
\begin{align*}
C_{S_2}(z;q)&=\frac{1}{(q)_N}+(1-z)\sum_{n=1}^{N}\left[\begin{matrix} N \\ n \end{matrix}\right]\frac{(-1)^n(q)_nq^{n(n+1)/2}}{(q)_{n+N}}\left(\frac{1}{1-zq^n}-\frac{1}{z-q^n}\right)\\
&=\frac{1}{(q)_N}\left(1+\sum_{n=1}^{N}\frac{(-1)^n (q)^2_N q^{\frac{n(n+1)}{2}}}{(q)_{N+n}(q)_{N-n}}\left( \frac{1-z}{1-zq^n}+\frac{1-z^{-1}}{1-z^{-1}q^n}\right)\right).
\end{align*}
Making a change of variable as in Theorem \ref{gffinsymmetrizedrankfn}, we finally get
\begin{equation*}
 C_{S_2}(z;q)=\frac{1}{(q)_N}\sum_{n=-N}^{N}\frac{(-1)^n (q)^2_N q^{\frac{n(n+1)}{2}}}{(q)_{N+n}(q)_{N-n}} \left(\frac{1-z}{1-zq^n}\right).
\end{equation*}
Hence, for $j \geq 1$, we have
\begin{align*}\label{crankgfderivatives}
C^{(j)}_{S_2}(z;q)=\frac{-j!}{(q)_N}\sum_{\substack {n=-N\\n\neq 0}}^{N}\frac{(-1)^n (q)^2_N q^{\frac{n(n-1)}{2}+jn}  }{(q)_{N+n}(q)_{N-n}} \left(\frac{1-q^n}{(1-zq^n)^{j+1}}\right).
\end{align*}
Substituting these derivative expressions in \eqref{gfmuintermsofCS2} and then by an application of binomial theorem, we obtain
\begin{equation*}
\sum_{n=1}^{\infty}\mu_{2\nu,N}(n)q^n=(q)_N\sum_{\substack {n=-N\\n\neq 0}}^{N}\frac{(-1)^{n-1}  q^{\frac{n(n+1)}{2}+\nu n}  }{(q)_{N+n}(q)_{N-n}(1-q^n)^{2\nu}}.
\end{equation*}
Splitting the sum into the ranges $1$ to $N$ and $-N$ to $-1$ and then making a variable change, we get
\begin{align*}{}
\sum_{n=1}^{\infty}\mu_{2\nu,N}(n)q^n=(q)_N \sum_{\substack {n=1}}^{N}\frac{(-1)^{n-1}  q^{\frac{n(n-1)}{2}+\nu n}(1+q^n)  }{(q)_{N+n}(q)_{N-n}(1-q^{n})^{2\nu}}.
\end{align*}
\end{proof}

\section{Proof of Theorem \ref{difference_gen_fin_rank_crank} and Conjecture \ref{finiteanalogcrankrankconjecture}}\label{baily_pairs_proof_conj}

Using Bailey's lemma, i.e., Theorem \ref{Baileylemma},  we give a result which is essential for the proof of Conjecture  \ref{finiteanalogcrankrankconjecture}. 
\begin{proposition}\label{generalBaileypair}
Let $(\alpha_n(a, q), \beta_n(a, q))$ be a Bailey pair with $a=1$ and $\alpha_0=\beta_0=1$. We then have
\begin{align*}
\sum_{N\geq n_k\geq ...\geq n_1\geq1} \frac{(q)^2_{n_1}q^{n_1+...+n_k} \beta_{n_1}}{(1-q^{n_k})^2 (1-q^{n_{k-1}})^2...(1-q^{n_1})^2}&=\sum_{N\geq n_k\geq ...\geq n_1\geq1}\frac{q^{n_1+...+n_k} }{(1-q^{n_k})^2 (1-q^{n_{k-1}})^2...(1-q^{n_1})^2} \nonumber \\ 
&+\sum_{r=1}^{N}\frac{(q)^2_N}{(q)_{N-r}(q)_{N+r}}\frac{q^{kr}\alpha_r}{(1-q^r)^{2k}}.
\end{align*}

\end{proposition}

\begin{proof}
Since $(\alpha_n(a, q), \beta_n(a, q))$ form a Bailey pair with $a=1$, we have the relation   
$$\beta_n(1, q) = \sum_{r=0}^{n}\frac{\alpha_r(1, q)}{(q)_{n-r} (q)_{n+r}}.$$
By Bailey's Lemma, $(\alpha'_n(a, q), \beta^{'}_n(a, q))$ is also a Bailey pair with parameters $(1, q)$. Hence, by \eqref{Baileydefn},
$$
\beta^{'}_n(1, q) = \sum_{r=0}^{n}\frac{\alpha'_r(1, q)}{(q)_{n-r} (q)_{n+r}}.
$$
Substituting the values of $\alpha'_n(a, q)$ and $\beta^{'}_n(a, q)$ from Theorem \ref{Baileylemma}, we get
$$
\sum_{k=0}^{n}\frac{(\rho_1)_k (\rho_2)_k (q/\rho_1\rho_2)_{n-k}}{(q/\rho_1)_n (q/\rho_2)_n (q)_{n-k}}\left(\frac{q}{\rho_1 \rho_2}\right)^k\beta_k(1, q)
=\sum_{k=0}^{n}   \frac{(\rho_1)_k (\rho_2)_k}{(q/\rho_1)_k (q/\rho_2)_k(q)_{n-k} (q)_{n+k}}\left(\frac{q}{\rho_1\rho_2} \right)^k \alpha_k(1, q).
$$
Separating the terms corresponding to $k=0$ in both the summations and multiplying throughout by $(q/\rho_1)_n (q/\rho_2)_n$,
\begin{align*}
\sum_{k=1}^{n}\frac{(\rho_1)_k(\rho_2)_k (q/\rho_1\rho_2)_{n-k} (q/\rho_1 \rho_2)^k}{(q)_{n-k}}\beta_k(1, q)&=\frac{(q/\rho_1)_n (q/\rho_2)_n}{(q)_{n}^2}\left(1-\frac{(q)_n(q/\rho_1\rho_2)_n}{(q/\rho_1)_n (q/\rho_2)_n}\right) +  \\
& (q/\rho_1)_n (q/\rho_2)_n \sum_{k=1}^{n}\frac{(\rho_1)_k (\rho_2)_k(q/\rho_1 \rho_2)^k}{(q/\rho_1)_k (q/\rho_2)_k(q)_{n-k} (q)_{n+k}} \alpha_k(1, q).
\end{align*}
Dividing both sides by $(1-\rho_1)(1-\rho_2)$, then letting $\rho_1\rightarrow1$, $\rho_2\rightarrow1$ and using \eqref{limitresultp1p2}, we get
$$
\sum_{k=1}^{n}(q)^2_{k-1}q^k \beta_k=\sum_{k=1}^{n}\frac{q^k}{(1-q^k)^2}+\sum_{k=1}^{n}\frac{(q)^2_{n}q^{k}\alpha_k}{(q)_{n-k}(q)_{n+k}(1-q^k)^2}.
$$
This is the $k=1$ case of the theorem. We are going to prove the theorem using induction. To this end, suppose that the theorem holds for $k=\ell-1$.
This means that
\begin{align}\label{l-1induction}
\sum_{N\geq n_\ell\geq ...\geq n_2\geq1} \frac{(q)^2_{n_2}q^{n_2+...+n_\ell} \beta_{n_2}}{(1-q^{n_2})^2...(1-q^{n_\ell})^2}=\sum_{N\geq n_\ell\geq ...\geq n_2\geq1}\frac{q^{n_2+...+n_\ell} }{(1-q^{n_2})^2...(1-q^{n_\ell})^2} \nonumber \\ 
+\sum_{r=1}^{N}\frac{(q)^2_N}{(q)_{N-r}(q)_{N+r}}\frac{q^{(\ell-1)r}\alpha_r}{(1-q^{r})^{2(\ell-1)}}.
\end{align}
This equation is true for any Bailey pair $(\alpha_n(a, q), \beta_n(a, q))$ with $a=1$ and $\alpha_0=\beta_0=1$. Note that, since 
$\alpha'_0=\alpha_0=1$ and $\beta^{'}_0=\beta_0=1$, \eqref{l-1induction} also holds for the Bailey pair $(\alpha'_n(a, q), \beta^{'}_n(a, q))$. So, we replace
$(\alpha_n, \beta_n)$ by $(\alpha'_n, \beta^{'}_n)$ in \eqref{l-1induction} to get
\begin{align*}
\sum_{N\geq n_\ell\geq ...\geq n_2\geq1} \frac{(q)^2_{n_2}q^{n_2+...+n_\ell} \beta^{'}_{n_2}}{(1-q^{n_2})^2...(1-q^{n_\ell})^2}=\sum_{N\geq n_\ell\geq ...\geq n_2\geq1}\frac{q^{n_2+...+n_\ell} }{(1-q^{n_2})^2...(1-q^{n_\ell})^2} \nonumber \\ 
+\sum_{r=1}^{N}\frac{(q)^2_N}{(q)_{N-r}(q)_{N+r}}\frac{q^{(\ell-1)r}\alpha'_r}{(1-q^{r})^{2(\ell-1)}}.
\end{align*}
We now substitute for $\alpha'_n$ and $\beta^{'}_n$ in terms of $\alpha_n$ and $\beta_n$ using Bailey's Lemma,
\begin{align*}
&\sum_{\substack{N\geq n_\ell\geq ...\geq n_2\geq1,\\n_2\geq n_1\geq 0}} \frac{(q)^2_{n_2}q^{n_2+...+n_\ell}}{(1-q^{n_2})^2...(1-q^{n_\ell})^2}\frac{(\rho_1)_{n_1}(\rho_2)_{n_1}(q/\rho_1\rho_2)_{n_2-n_1}(q/\rho_1\rho_2)^{n_1}\beta_{n_1}}{(q/\rho_1)_{n_2}(q/\rho_2)_{n_2}(q)_{n_2-n_1}}\\
&=\sum_{N\geq n_\ell\geq ...\geq n_2\geq1}\frac{q^{n_2+...+n_\ell} }{(1-q^{n_2})^2...(1-q^{n_\ell})^2}+\sum_{r=1}^{N}\frac{(q)^2_N}{(q)_{N-r}(q)_{N+r}}\frac{q^{(\ell-1)r}(\rho_1)_r(\rho_2)_r(q/\rho_1\rho_2)^r\alpha_r}{(1-q^{r})^{2(\ell-1)}(q/\rho_1)_r(q/\rho_2)_r}.
\end{align*}
Again, separating the terms corresponding to $n_1=0$ from the sum on the left side, then dividing both sides by $(1-\rho_1)(1-\rho_2)$, letting $\rho_1\rightarrow1$, $\rho_2\rightarrow1$ and using \eqref{limitresultp1p2}, we obtain
\begin{align*}
\sum_{N\geq n_\ell\geq ...\geq n_1\geq1} \frac{(q)^2_{n_1}q^{n_1+n_2+...+n_\ell} \beta_{n_1}}{(1-q^{n_1})^2(1-q^{n_2})^2...(1-q^{n_\ell})^2}=&\sum_{N\geq n_\ell\geq ...\geq n_2\geq n_1\geq1}\frac{q^{n_1+n_2+...+n_\ell}}{(1-q^{n_1})^2(1-q^{n_2})^2...(1-q^{n_\ell})^2}. \nonumber \\ 
&+\sum_{r=1}^{N}\frac{(q)^2_Nq^{\ell r}\alpha_r}{(q)_{N-r}(q)_{N+r}(1-q^{r})^{2\ell}}.
\end{align*}
This concludes the proof of the theorem by induction.
\end{proof}

\begin{corollary}\label{mugfcorollary}
\begin{equation*}
(q)^2_N\sum_{r=1}^{N}\frac{(-1)^{r-1}  q^{\frac{r(r-1)}{2}+kr} (1+q^r)}{(q)_{N-r}(q)_{N+r}(1-q^{r})^{2k}}=
\sum_{N\geq n_k\geq ...\geq n_1\geq1}\frac{q^{n_1+n_2+...+n_k} }{(1-q^{n_1})^2(1-q^{n_2})^2...(1-q^{n_k})^2}.
\end{equation*}
\end{corollary}
\begin{proof}
Consider the well known Bailey pair below (\cite[pp. ~27-28]{andrewsqseries}),
\begin{equation*}
\alpha_n=
\begin{cases}
1, &\; \text{if} \;\;n=0, \\
(-1)^n q^{\frac{n(n-1)}{2}} (1+q^n),& \;\text{if}\;\;n\geq1,
\end{cases}
\end{equation*}
and
\begin{equation*}
\beta_{n}=
\begin{cases}
1, &\; \text{if} \;\;n=0, \\
0,& \;\text{if}\;\;n\geq1.
\end{cases}
\end{equation*}
Substituting the above Bailey pair in Theorem \ref{generalBaileypair}, we get
\begin{align*}
0=\sum_{N\geq n_k\geq ...\geq n_1\geq1}\frac{q^{n_1+n_2+...+n_k} }{(1-q^{n_1})^2(1-q^{n_2})^2...(1-q^{n_k})^2}
+(q)^2_N\sum_{r=1}^{N}\frac{(-1)^r q^{\frac{r(r-1)}{2}+kr} (1+q^r)}{(q)_{N-r}(q)_{N+r}(1-q^{r})^{2k}}.
\end{align*}
\end{proof}
\begin{corollary}
\begin{align}\label{mugfalternate}
\sum_{n=1}^{\infty}\mu_{2k,N}(n)q^n=
\frac{1}{(q)_N}\sum_{N\geq n_k\geq ...\geq n_1\geq1}\frac{q^{n_1+n_2+...+n_k} }{(1-q^{n_1})^2(1-q^{n_2})^2...(1-q^{n_k})^2}.
\end{align}
\end{corollary}
\begin{proof}
Using equation \eqref{mugf} from Theorem \ref{theoremmugf} along with Corollary \ref{mugfcorollary}, we get this result. Note that this is a finite analogue of \eqref{mugf_alternate}. 
\end{proof}

\begin{corollary}\label{crudediffsymmoments}
\begin{align*}
&\sum_{N\geq n_k\geq ...\geq n_1\geq1} \frac{(q)_{n_1}q^{n_1+n_2+...+n_k}}{(1-q^{n_1})^2(1-q^{n_2})^2...(1-q^{n_k})^2}= \nonumber\\
&\sum_{N\geq n_k\geq ...\geq n_1\geq1}\frac{q^{n_1+n_2+...+n_k} }{(1-q^{n_1})^2(1-q^{n_2})^2...(1-q^{n_k})^2}
+(q)^2_N\sum_{r=1}^{N}\frac{(-1)^r q^{\frac{r(3r-1)}{2}+kr} (1+q^r)}{(q)_{N-r}(q)_{N+r}(1-q^{r})^{2k}}.
\end{align*}
\end{corollary}

\begin{proof}
Again we use a well known Bailey pair (\cite[p.~28]{andrewsqseries}),
\begin{equation*}
\alpha_n=
\begin{cases}
1, &\; \text{if} \;\;n=0, \\
(-1)^n q^{\frac{n(3n-1)}{2}} (1+q^n),& \;\text{if}\;\;n\geq1,
\end{cases}
\ \ \text{and} 
\ \ \beta_n=\frac{1}{(q)_n}.
\end{equation*}
Putting the values of $\alpha_n$ and $\beta_n$ in Theorem \ref{generalBaileypair}, we get the result.
 \end{proof}
Now we are ready to give a proof of Theorem \ref{difference_gen_fin_rank_crank}.
\begin{proof}[Theorem \textup{\ref{difference_gen_fin_rank_crank}}][]
Divide both sides of \eqref{crudediffsymmoments} by $(q)_N$ to get 

\begin{align*}
&\frac{1}{(q)_N}\sum_{N\geq n_k\geq ...\geq n_1\geq1}\frac{(q)_{n_1}q^{n_1+n_2+...+n_k}}{(1-q^{n_1})^2(1-q^{n_2})^2...(1-q^{n_k})^2}=\\
&\frac{1}{(q)_N}\sum_{N\geq n_k\geq ...\geq n_1\geq1}\frac{q^{n_1+n_2+...+n_k} }{(1-q^{n_1})^2(1-q^{n_2})^2...(1-q^{n_k})^2}
+(q)_N\sum_{r=1}^{N}\frac{(-1)^r q^{\frac{r(3r-1)}{2}+kr} (1+q^r)}{(q)_{N-r}(q)_{N+r}(1-q^{r})^{2k}}.
\end{align*}
Using \eqref{mugfalternate} and equation \eqref{gfsymrankmoment} from Theorem \ref{gffinsymmetrizedrankfn}, we get the desired result.
\end{proof}

Before going to the proof of Conjecture \ref{finiteanalogcrankrankconjecture}, we require one more concept, an analogue of Stirling numbers of the second kind, defined by Garvan \cite{garvan11}. 
He defined a sequence of polynomials
\begin{align*}
g_k(x)=\prod_{j=0}^{k-1}(x^2-j^2), \;\;\; \text{for} \;\;k\geq1
\end{align*}
and a sequence of numbers $S^*(n,k)$ such that, for $n\geq1$,
\begin{equation}\label{polynomialSnk}
x^{2n}=\sum_{k=1}^{n}S^*(n,k)g_k(x).
\end{equation}
\textbf{Definition} \cite[p.~249]{garvan11}: Define the sequence $S^*(n,k)$, for $1 \leq k \leq n$, recursively by\\
(i) $S^*(1,1)=1$,\\
(ii) $S^*(n,k)=0$ if $k\leq0$ or $k>n$,\\
(iii) $S^*(n+1,k)=S^*(n,k-1)+k^2S^*(n,k)$, for $1\leq k\leq n+1$.\\
From this definition, Garvan showed that the relation \eqref{polynomialSnk} indeed holds (\cite[Lemma 4.2]{garvan11}).
Next, we link the finite analogues of rank and crank moments with their symmetrized counterparts via the numbers $S^*(n,k)$.
\begin{proposition}
For any two positive integers $k$ and $N$,
\begin{align}
\mu_{2k,N}(n)&=\frac{1}{(2k)!}\sum_{m=-n}^{n}g_k(m)M_{S_2}(m,n),\label{fg1}\\
\eta_{2k,N}(n)&=\frac{1}{(2k)!}\sum_{m=-n}^{n}g_k(m)N_{S_1}(m,n),\label{fg2}\\
M_{2k,N}(n)&=\sum_{j=1}^{k}(2j)!S^*(k,j)\;\mu_{2j,N}(n),\label{fg3}\\
N_{2k,N}(n)&=\sum_{j=1}^{k}(2j)!S^*(k,j)\;\eta_{2j,N}(n).\label{fg4}
\end{align}
\end{proposition}
\begin{proof}
By the definition of finite analogue of $k^{th}$ symmetrized crank moment, we know that
\begin{align*}
\mu_{2k,N}(n)&=\sum_{m=-n}^{n} \left(\begin{matrix} m+k-1 \\ 2k\end{matrix}\right)M_{S_2}(m,n)\\
&=\frac{1}{(2k)!}\sum_{m=-n}^{n}\left(m^2-(k-1)^2\right)
\left(m^2-(k-2)^2\right)
\dots\big(m^2- 1^2\big)
m(m-k)M_{S_2}(m,n).
\end{align*} 
By the definition of the polynomials $g_k$, this may be written as
\begin{align*}
\mu_{2k,N}(n)=\frac{1}{(2k)!}\sum_{m=-n}^{n}g_k(m)M_{S_2}(m,n)-\frac{k}{(2k)!}\sum_{m=-n}^{n}\left(m^2-(k-1)^2\right)\dots\big(m^2-1^2\big)mM_{S_2}(m,n).
\end{align*}
Since $M_{S_2}(m,n)=M_{S_2}(-m,n)$ \cite[p.~9]{DEMS}, the rightmost sum vanishes and we get \eqref{fg1}. Similarly one can prove \eqref{fg2}.
Now for the proof of \eqref{fg3}, we start with the definition of $M_{2k,N}(n)$ \eqref{fincrankmom}, namely, 
\begin{align*} 
M_{2k,N}(n)=\sum_{m=-n}^{n}m^{2k}M_{S_2}(m,n).
\end{align*}
We use \eqref{polynomialSnk} to substitute for $m^{2k}$ and obtain
\begin{align*} 
M_{2k,N}(n)&=\sum_{m=-n}^{n}\left(\sum_{j=1}^kS^*(k,j)g_j(m)\right)M_{S_2}(m,n)\\
&= \sum_{j=1}^{k}S^*(k,j)\sum_{m=-n}^{n}g_j(m)M_{S_2}(m,n)\\
&=\sum_{j=1}^k(2j)!S^*(k,j)\mu_{2j,N}(n),
\end{align*}
the last step following from \eqref{fg1}. This completes the proof of \eqref{fg3} and on similar lines we can prove \eqref{fg4}.
\end{proof}
We are now ready to prove the inequality for the finite analogues of rank and crank moments.
\begin{proof}[Conjecture \textup{\ref{finiteanalogcrankrankconjecture}}][]
From \eqref{fg3} and \eqref{fg4}, we get 
\begin{equation}\label{diffmoments}
M_{2k,N}(n)-N_{2k,N}(n)=\sum_{j=1}^{k}(2j)!S^*(k,j)\;(\mu_{2j,N}(n)-\eta_{2j,N}(n)).
\end{equation}
From Theorem \ref{difference_gen_fin_rank_crank}, we know
$$\sum_{n=1}^{\infty}(\mu_{2t,N}(n)q^n-\eta_{2t,N}(n))q^n=\sum_{N\geq n_t\geq ...\geq n_1\geq1}\frac{q^{n_1+n_2+...+n_t}}{(1-q^{n_1})^2(1-q^{n_2})^2...(1-q^{n_t})^2(q^{n_1 + 1})_{N-n_1}}.$$
From the generating function, we infer that, $\mu_{2t,N}(n)-\eta_{2t,N}(n)\geq 0$ for $n, t, N \geq 1$. Moreover, the numbers $S^*(k,j)$ are all positive, so from \eqref{diffmoments}, we can write
$$M_{2k,N}(n)-N_{2k,N}(n) \geq 2(\mu_{2,N}(n)-\eta_{2,N}(n)) = 2\textup{spt}(n, N)>0,$$
where the last equality follows from Remark \ref{remark_fin_higher_spt}. This finishes the proof of the conjecture. 
\end{proof}

 \section{Concluding Remarks}
In Remark \ref{remark_fin_higher_spt}, we defined a finite analogue of higher order spt-function as 
$\textup{spt}_k(n,N):= \mu_{2k,N}(n)-\eta_{2k,N}(n)$.
A combinatorial interpretation of the higher order spt-function $\text{spt}_{k}(n)$ was described by Garvan \cite[p.~252]{garvan11}. Looking at the generating function \eqref{gen_fin_rank_crank} of the difference between finite analogues of the symmetrized moments and comparing it with \eqref{difference_gen_garvan}, one can give a combinatorial interpretation of $\text{spt}_k(n,N)$ on similar lines as that of Garvan's for $\text{spt}_k(n)$, the only restriction being that the largest parts of the corresponding partitions are less than or equal to $N$. 

Bringmann, Mahlburg and Rhoades \cite{BMR11} showed that, for any $k \geq 1$, as $n \rightarrow \infty$, 
\begin{align*}
& M_{2k}(n) \sim  N_{2k}(n) \sim \alpha_{2k} n^k p(n),\\
& M_{2k}(n)-N_{2k}(n) \sim \beta_{2k} n^{k- \frac{1}{2}} p(n),
\end{align*}
where $\alpha_{2k}, \beta_{2k}$ are certain explicitly computable constants (see \cite[p. 665, Corollary 1.4]{BMR11}). Since in this paper, we have proved the inequality for the finite analogues of rank and crank moments, it would be fascinating to find the asymptotic behavior of the finite analogues and their difference.

 Given any prime $p>3$ and for fixed positive integers $k$ and $j$,  Bringmann, Garvan and Mahlburg \cite[Corollary 1.3]{BGM09} established that there are infinitely many arithmetic progressions $A n + B$ such that
$\eta_{2k}(An+B) \equiv 0 \pmod {p^j}$.
It would be worthwhile to see if such congruences exist for $\eta_{2k,N}(n)$.

A number of explicit congruences for higher order $\text{spt}$-functions were proved by Garvan \cite[Theorem 6.1--6.3]{garvan11}. It would also be interesting to see if there exists a refinement of these congruences for $\text{spt}_{k}(n,N)$.

\textbf{Acknowledgements}
We would like to thank Prof.~Atul Dixit for going through the manuscript and giving valuable suggestions. The first author wishes to thank Harish-Chandra Research Institute and IIT Gandhinagar for the conducive environment. The second author is a SERB  National Post Doctoral Fellow (NPDF) supported by the fellowship PDF/2017/000370. The third author is supported partially by IIT Gandhinagar and by SERB ECR grant ECR/2015/000070 of Prof. Atul Dixit.


\begin{thebibliography}{00}

%
%
%
%
%
%




\bibitem{andrewsqseries}
G.~E.~Andrews, q-series: their development and application in analysis, number theory, combinatorics, physics, and computer algebra, CBMS Regional Conference Series in Mathematics, 66, 
American Mathematical Society, Providence, RI, 1986.

\bibitem{andpar}
G.~E.~Andrews, \emph{Ramanujan and partial fractions}, Contributions to the History of Indian Mathematics, Hindustan Book Agency, New Delhi, 2005.

\bibitem{andrews07}G.E. Andrews, \emph{Partitions, Durfee symbols, and the Atkin-Garvan moments of ranks}, Invent. Math. 169 (2007) 37--73.

\bibitem{andrews08} G.~E.~Andrews, \emph{The number of smallest parts in the partitions on $n$}, J. Reine Angew. Math. {\bf 624} (2008), 133--142.


\bibitem{andrewsaskey99} G.~E.~Andrews, R.~Askey, R.~Roy, Special functions. Encyclopedia of Mathematics and its Applications, 71. Cambridge University Press, Cambridge, 1999
 
\bibitem{abramlostI} G.~E.~Andrews, B.~C.~Berndt, Ramanujan's Lost Notebook, Part I, Springer, New York, 2005. 




\bibitem{andrewsgarvan88} G.~E.~Andrews and F.~G. ~Garvan, \emph{Dyson's crank of a partition}, Bull. Amer. Math. Soc. (N.S.) {\bf 18} (1988) 167--171.




\bibitem{atkin1} A. O. L. Atkin and F. Garvan, \emph{Relations between the ranks and cranks of partitions}, Ramanujan J., 7 (2003), 343--366.

\bibitem{atkin2} A.~O.~L.~Atkin and H.~P.~F. ~Swinnerton-Dyer, \emph{Some properties of partitions}, Proc. London Math. Soc., Ser. 3, 4 (1954), 84--106.




%
\bibitem{bcbramforthnote}
B.~C.~Berndt, Ramanujan's Notebooks, Part IV, Springer-Verlag, New York, 1991.



%
%

\bibitem{BM09} K.~Bringmann, K.~Mahlburg, \emph{Inequalities between ranks and cranks}, Proc.~Amer.~Math.~Soc., {\bf 137},  2009, 2567--2574.

\bibitem{BGM09} K.~Bringmann, F.~Garvan, and K.~Mahlburg, \emph{Partition statistics and quasiharmonic Maass forms}, Int.~Math.~Res.~Not. IMRN 2009 (1) (2009) 63--97.


\bibitem{BMR11} K.~Bringmann, K.~Mahlburg, R.~C.~Rhoades, \emph{Asymptotics for rank and crank moments}, Bull. London Math. Soc.,  {\bf 43} (2011) 661--672.


%
%
%


\bibitem{dixitmaji18} A.~Dixit and B.~Maji, \emph{Partition implications of a three parameter $q$-series identity}, to appear in Ramanujan J., arXiv:1806.04424

\bibitem{DEMS} A.~Dixit, P.~Eyyunni, B.~Maji and G.~Sood, \emph{Untrodden pathways in the theory of the restricted partition function p(n,N)}, arXiv:1812.01424, 2018.

\bibitem{dys}F.J.~Dyson, \emph{Some guesses in the theory of partitions}, Eureka (Cambridge) \textbf{8} (1944), 10--15.

\bibitem{dyson89} F.~J.~Dyson, \emph{Mappings and symmetries of partitions}, J. Combin. Theory Ser. A {\bf 51} (1989) 169--180.

\bibitem{fine} N. J. Fine, Basic Hypergeometric Series and Applications, Mathematical Surveys and Monographs, Amer. Math. Soc., 1989. 



\bibitem{garvan88} F.~G.~Garvan, \emph{New combinatorial interpretations of Ramanujan's partition congruences mod $5, 7$ and $11$}, Trans.~ Amer.~Math.~Soc. {\bf 305} (1988) 47--77.

\bibitem{garvan10} F.~G.~Garvan, \emph{Congruences for Andrews' smallest parts partition function and new congruence for Dyson's rank}, Int.~J.~Number~Theory~, {\bf 6}, (2010) 281--309.

\bibitem{garvan11} F.~G.~Garvan, \emph{Higher order spt-function}, Adv.~Math. {\bf 228} (2011) 241--265.






%
%
%
%



%

%
%
\bibitem{ramanujantifr}
S.~Ramanujan, Notebooks of Srinivasa Ramanujan, Vol. II, Tata Institute of Fundamental Research, Mumbai, 2012.


%
%
%
%
%


%
%
%
%


\end{thebibliography}
\end{document}